\documentclass[12pt]{amsart}
\usepackage{fullpage}
\usepackage{amsmath,amssymb,amscd,amsthm,amsfonts}
\usepackage{graphicx}
\usepackage{hyperref}
\usepackage{pstricks}
\usepackage[section]{placeins}
\usepackage{parskip}
\usepackage[colorinlistoftodos,bordercolor=orange,backgroundcolor=orange!20,linecolor=orange,textsize=scriptsize]{todonotes}
\usepackage{amsthm}
\theoremstyle{plain}

\newtheorem*{theorem*}{Theorem}
\newtheorem{theorem}{Theorem}

\newtheorem{definition}{Definition}
\newtheorem{lemma}{Lemma}

\newtheorem{claim}[theorem]{Claim}

\newcommand{\p}{\boldsymbol{p}}
\newcommand{\x}{\boldsymbol{x}}

\newcommand{\bs}{\boldsymbol{s}}
\newcommand{\R}{\mathbb{R}}

\newcommand{\N}{\mathcal{N}}
\newcommand{\s}{\mathcal{S}}

\usepackage{subfiles}

\DeclareMathOperator{\cl}{closure}
\DeclareMathOperator{\conv}{conv}

\DeclareMathOperator{\Tv}{Tv}

\title{Tverberg-Type Theorems with Trees and Cycles \\as (Nerve) Intersection Patterns.}

\author[1]{J.A. De Loera}
\author[2]{T. Hogan}
\author[3]{D. Oliveros}
\author[4]{D. Yang}
\address[J. A. De Loera, T. A. Hogan, D. Yang]{Department of Mathematics, University of California, Davis}
\email{deloera@math.ucdavis.edu, tahogan@math.ucdavis.edu, domyang@ucdavis.edu}
\address[D. Oliveros]{Instituto de Matem\'aticas, Universidad Nacional Aut\'onoma de M\'exico.}
\email{dolivero@matem.unam.mx}
\begin{document}

\maketitle
\begin{abstract}
Tverberg's theorem says that a set with sufficiently many points in $\R^d$ can always be partitioned into $m$ parts so that the
$(m-1)$-simplex is the (nerve) intersection pattern of the convex hulls of the parts. The main results of our paper demonstrate that, 
Tverberg's theorem is but a special case of a more general situation. Given sufficiently many points, all trees and cycles can also 
be induced by at least one partition of a point set.  

\end{abstract}

\noindent {\em Math Subject Classification:} {Primary 90C15, 90C11, 90C48. Secondary 52A35, 52A01}
\section{Introduction}

\noindent The celebrated theorem of Helge Tverberg states (see \cite{TvSurvey} and the references therein): 

\begin{theorem*}[Tverberg 1966 \cite{tv}] Every set $S$ with at least $(d+1)(m-1)+1$ points in Euclidean $d$-space $\R^d$ can 
be partitioned into $m$ parts $\mathcal{P}=S_1, \dots, S_m$  such that all the convex hulls of these parts have nonempty intersection. 
The special case of a bi-partition, $m=2$, is called \emph{Radon's lemma}.
\end{theorem*}

The nerve (intersection pattern) of the convex hulls in Tverberg's theorem is very specific, a simplex; our paper investigates other possible nerves. 
Informally, the main results of our paper demonstrate that, given sufficiently many points, other kinds of nerves can always be induced by a suitable 
partition of the point set. In particular, we show that any tree or cycle can be induced as the nerve. We will see this depends on a universal constant depending only on the dimension too.

To state our results precisely we begin with some terminology and notation typical of geometric topological combinatorics (see \cite{Mbook,Tancer} 
for details, especially on simplicial complexes discussed here). Let $\mathcal{F} = \{F_1, \dots, F_m\}$ be a family of convex sets in $\R^d$.  
The \emph{nerve}  $\mathcal{N}(\mathcal{F})$ of $\mathcal{F}$ is the simplicial complex  with vertex set $[m]:=\{1,2\dots , m\}$ whose faces are 
$I \subset [m]$ such that $\cap_{i \in I} F_i \neq \emptyset$.


\sloppy Given a collection of points $S \subset \R^d$ and an $n$-partition into $n$ color classes \allowbreak$\mathcal{P} = S_1, \dots, S_n$ of $S$, 
we define \emph{the nerve of the partition},  $\mathcal{N}(\mathcal{P})$  to be the nerve complex  \allowbreak$\mathcal{N}(\{\conv(S_1), \dots, \conv(S_n)\})$, 
where $\conv(S_i)$ is the convex hull of the elements in the color class $i$. Similarly, given a partition $\mathcal{P}$, we define the 
\emph{intersection graph of the partition}, denoted $\mathcal{N}^1(\mathcal{P})$, as the $1$-skeleton of the nerve of $\mathcal{P}$. 

Given a simplicial complex $K$, and a finite set of points $S$ in $\R^d$, we say that $K$ is \emph{partition induced on $S$} if there exists a partition 
$\mathcal{P}$ of $S$ such that the nerve of the partition is isomorphic to $K$. We say that $K$ is \emph{$d$-partition induced} if there exists at least 
one set of points $S \subset \R^d$ such that $K$ is partition induced on $S$. 
It was shown by G. Y. Perelman \cite{Perelman} that every $d$-dimensional simplicial complex is $(2d+1)$-partition induced on some point set.
This result is in fact optimal, because the barycentric subdivision of the $d$-skeleton of a $(2d + 2)$-dimensional simplex
is not $2d$-partition induced, see \cite{WegnerPhD} and \cite{Tancer2} for details.

Motivated by Tverberg's theorem, we introduce another property of simplicial complexes that is much stronger than being $d$-partition induced 
because it has to hold in all point sets once they have sufficiently many points.

\begin{definition}\label{deftv}
A simplicial complex $K$ is \emph{$d$-Tverberg} if there exists a constant $\Tv(K,d)$ such that $K$ is 
partition induced on all point sets $S \subset \R^d$ in general position with $|S| > \Tv(K,d)$. 
The minimal such constant $\Tv(K,d)$ is called the \emph{Tverberg number for $K$ in dimension $d$}. 
\end{definition}

Let us briefly examine the definition of $d$-Tverberg complexes. First of all, note one can re-state the classical Tverberg's theorem as follows:
 
\begin{theorem*}[Tverberg's theorem rephrased]
The $(m-1)$-simplex is a  $d$-Tverberg complex for all $d\geq 1$, with Tverberg number $(d+1)(m-1)+1$.
\end{theorem*}

Definition \ref{deftv} can be compared with earlier work by Reay and others \cite{Reay1979, ROUDNEFF2009, Perles2016}, 
who asked what happens when we demand only that each $k$ of the convex hulls intersect. They looked for the smallest number $n$
of points sufficient so that some partition induces a nerve which contains the $(k-1)$-skeleton of a simplex. In fact, 
Reay's conjecture says for every $n \leq (d+1)(m-1)$  there exists an $n$ point set $X \subset \R^d$ such that no partition of $X$ 
induces the complete graph $K_m$ as its intersection graph. In contrast, we are looking for an exact nerve of general kind.

Definition \ref{deftv} is most interesting for sets $S \subset \R^d$ in general position.  The reason is that for collinear 
points the only type of nerve complexes possible are those whose graphs are \emph{interval graphs}. Interval graphs have 
been  classified~\cite{Lekkeikerker} and  in particular are \emph{chordal}. With Definition \ref{deftv} the $4$-cycle graph is 
not $1$-Tverberg, because it is not chordal, but we will show later that it is  $d$-Tverberg for all $d\geq 2$. Similarly, while 
every $d$-Tverberg complex $K$ is clearly $d$-partition induced,  the converse is not true. The complex in Figure~\ref{badexample} 
is a graph that is partition induced on some planar point sets, but not for points in convex position, regardless of how many 
points we use. Thus it is not a $2$-Tverberg complex. Details are presented in Section~\ref{appendix}.

\begin{center}
\begin{figure}[h]~\label{badexample}
\centering \includegraphics[scale=.60]{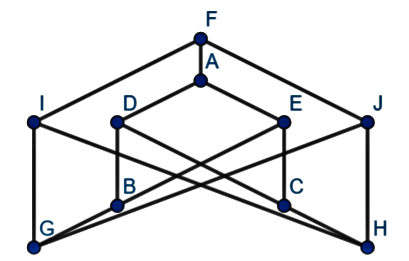}
\caption{A $2$-partition induced one-dimensional complex that is not $2$-Tverberg}
\end{figure}
\end{center}

The key contribution of our paper is to generalize the classical Tverberg's theorem by showing that similar theorems exist where other simplicial 
complexes -not just simplices- are $d$-Tverberg complexes too.  Before stating our first result, recall that the \emph{$k$-hypergraph Ramsey number} 
$R_k(m)$ is the least integer $N$ such that every red-blue $2$-coloring of all $k$-subsets of an $N$-element set contains either a 
red set of size $m$ or a blue set of size $m$, where a set is called red (blue) if all $k$-subsets from this set are red (or respectively blue). 
See \cite{CDFhypergraphramsey} and references therein.

\begin{theorem}\label{tverbertrees+cycles}
All trees and cycles are $d$-Tverberg complexes for all $d \geq 2$.
\quad \vskip 11pt 
\begin{enumerate}
\item[(A)] Every tree $T_n$ on $n$ nodes, is a $d$-Tverberg complex for $d\geq 2$. The Tverberg number $\Tv(T_n,d)$ exists and it is at most $R_{d+1}((d+1)(n-1)+1)$. 
More strongly, $\Tv(T_n,2)$ is at most  ${4n-4 \choose 2n-2 } + 1$.

\item[(B)]  Every $n$-cycle $C_n$ with $n \geq 4$ is a $d$-Tverberg complex for $d\geq 2$. The Tverberg number exists and  
$\Tv(C_n,d)$ is at most  $nd+n+4d$.
\end{enumerate}
\end{theorem}

The proof of Theorem \ref{tverbertrees+cycles} relies on several powerful non-constructive tools such as the
Ham-Sandwich theorem (see Section 1.3 \cite{Mbook}), a characterization of oriented matroids of cyclic polytopes \cite{CordovilDuchet}, 
and the multi-dimensional version of Erd\"os-Szekeres theorem (this is due to Gr\"unbaum \cite{Gbook} and 
Cordovil and Duchet \cite{CordovilDuchet},  see also Chapter 9 of \cite{OMbook}, and the survey \cite{Morris}). 
These tools are enough to show the existence of a Tverberg number $\Tv(T_n,d)$, but the bounds are  far from tight. 
Details are presented in Section \ref{tverbertreesANDcycles}.

We can  prove the following general  lower bound for the Tverberg numbers (see Appendix for the argument).

\begin{lemma}\label{lowerb}  For any connected simplicial complex $K$ with $n \geq 2$ vertices, if it exists, then  $\Tv(K,d) \geq 2n$.  \end{lemma}

In addition to this general lower bound, we show that the upper bounds of Theorem \ref{tverbertrees+cycles} can indeed be 
improved by giving better bounds on the Tverberg numbers of \emph{caterpillar trees}. Caterpillar trees are 
those in which all  the vertices are within distance one of a central path; these include paths and stars. See Section \ref{specialtrees}.

\begin{theorem} \label{improvedtrees}
If a tree $T_n$ is a caterpillar tree with $n$ nodes, then $T_n$ is $d$-Tverberg complex for all $d$, and its 
$d$-Tverberg number $\Tv(T_n,d)$ is no more than $(d+1)(n-1)+1$.
\end{theorem}

In terms of intersection properties
caterpillar graphs have been shown to be precisely the trees that are also interval graphs by Eckhoff \cite{EckhoffIntervalGraph}. 
In other words, the previous theorem implies that a tree $T_n$ is also $1$-Tverberg if and only if $T_n$ is a caterpillar tree.

Furthermore, in dimension two we can give some exact Tverberg numbers for trees:

\begin{theorem} \label{lowdim} 
\quad \vskip 10pt

\begin{enumerate}
\item[(A)]The $2$-Tverberg numbers $\Tv(S_n,2)$ for a star tree with $n$ nodes  equals $2n$.

\item[(B)] The $2$-Tverberg numbers of the path and cycle with four nodes are  $\Tv(P_4,2)=9$ and 
$11 \leq \Tv(C_4,2) \leq 13$.
\end{enumerate}
\end{theorem}

The proof of Theorem \ref{lowdim}  (B) requires exhaustive computer enumeration of all possible partitions, 
over all possible order types of point sets
with fewer than ten points. Luckily, these order types were classified in~\cite{Aichholzer}. 

Recall that for an ordered set of points $S = (\p_1, \p_2, \dots, \p_n)$ the \emph{order type} (see 9.3 \cite{Mbook}) of $S$ is defined  as the 
mapping assigning to each $(d+1)$-tuple $(i_1, i_2, \dots, i_{d+1})$ of indices, $1 < i_1 < i_2 < \dots <i_{d+1} \leq n$, the orientation of 
the $(d+1)$-tuple $(\p_{i_1}, \p_{i_2}, \dots , \p_{i_{d+1}})$ (i.e., the sign of the determinant of the corresponding matrix). The order type of $S$ 
is encoded by the \emph{chirotope} of $S$ which is the sequence of resulting  $\binom{n}{d+1}$ signs of possible determinants. This is a
vector of $+1$'s and $-1$'s, with $\binom{n}{d+1}$ entries.

The proof of Theorem \ref{lowdim}  (B)  also uses the following lemma to ensure that it suffices to check one representative configuration of points from each order type,
reducing calculations to finitely many cases. See details in the Appendix.

\begin{lemma}~\label{ordertype}
Suppose $S_1$ and $S_2$ are two point sets in $\R^d$ with the same order type, and let $\sigma$ be a bijection from $S_1$ to $S_2$ that 
preserves the orientation of any $(d+1)$-tuple in $S_1$. Then any partition $\mathcal{P} = (P_1, P_2, \dots, P_n)$ of $S_1$ and the 
corresponding partition of $S_2$ via $\sigma$, denoted $\sigma{\mathcal{P}} = \{\sigma(P_1), \sigma(P_2), \dots ,\sigma(P_n)\}$,  have 
the same intersection graph $\N^1(\mathcal{P})$.
\end{lemma}

Lemma~\ref{ordertype} cannot be extended to arbitrary nerve complexes as we see in the example of Figure \ref{tommy}.
Despite the fact that the chirotope-preserving bijections do not preserve the higher-dimensional skeleton of the nerve of a partition 
we can still make use of Lemma~\ref{ordertype} throughout our paper because our results are only about 
\emph{triangle-free} simplicial complexes, thus their nerve complexes equal their $1$-skeleton. 

\begin{center}
\begin{figure}[h]~\label{tommy}
\centering \includegraphics[scale=.50]{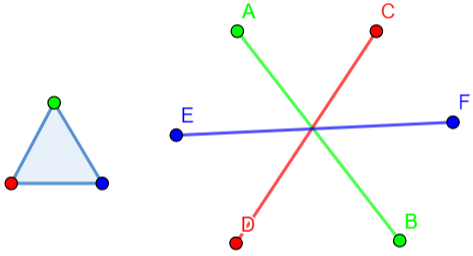} \hskip .5cm  \includegraphics[scale=.50]{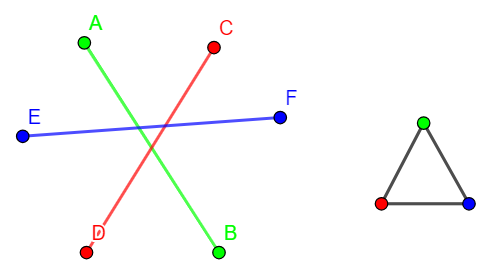}
\caption{Only the $1$-skeleton of the nerve is preserved by order-preserving bijection.}
\end{figure}
\end{center}


\section{A Tverberg theorem for Trees and Cycles}~\label{tverbertreesANDcycles}

\subsection{Proof of Theorem~\ref{tverbertrees+cycles} (A) in the plane}

Because the case of dimension two exemplifies the key ideas very well and because we can provide a 
better bound, we first give the proof of Theorem \ref{tverbertrees+cycles} (A) in the plane. 
To summarize the proof, first, we show in Theorem \ref{treeconvex} that the result holds if the points are arranged as the vertices of a convex polygon.
Second, given any set $\bar{S}$ with at least  ${4n-4 \choose 2n-2 } + 1$ points in the plane, we apply the Erd\H os-Szekeres theorem to deduce that $\bar{S}$ has a sub-configuration $S$ of $2n$ points in convex position. Then we apply Theorem~\ref{treeconvex} to obtain a partition of $S$ whose nerve is the tree $T_n$, and finally, in Lemma~\ref{extension2d}, we prove we can extend the partition of $S$ to the rest of $\bar{S}$ while preserving the nerve. 
Later in Subsection \ref{treesind} we  present the general case in $\R^d$ following a similar strategy, but some of the key steps are different.

\smallskip

\begin{theorem}\label{treeconvex}
Let $T_n$ be a tree with $n$ nodes, and let $S \subset \R^2$ be any $2n$ point set in convex position. 
Then $S$ admits a partition $\mathcal{P}$ such that its nerve $\N(\mathcal{P})$ is isomorphic to $T_n$.
\end{theorem}

\begin{proof}[\bf Proof]
The proof is by induction on $n$, the number of vertices in $T_n$. For an example of the 
construction see Figure \ref{example1}.

For $n = 1$, the tree consists of a single node and $S$ is a set of two points in $\R^2$. Coloring both points with color $1$ will trivially satisfy the theorem. When $n=2$, the only tree with two vertices is $K_2$. By Radon's theorem any set of four points in $S$, say $\boldsymbol{s}_1,\boldsymbol{s}_2,\boldsymbol{s}_3,\boldsymbol{s}_4$ in counterclockwise order, can be partitioned with intersection graph $K_2$. Note that in this case, coloring the points in $S=S_1\cup S_2$ with two alternating colors $\boldsymbol{s}_1=1,\bs_2=2, \bs_3=1, \bs_4=2$ will yield the required partition.

For performing the induction step, we can assume $T_n$ was obtained from a tree $T_{n-1}$ by adding the leaf node $v_n$ to a node $v_{r} \in T_{n-1}$ such that $\{v_n,v_r\}$ is an edge of $T_n$. Note that in our labeling of the $n$ nodes, $r$ may not be $n-1$, but all trees are constructed by a sequence of leaf additions.

By the induction hypothesis, for any set $S'$ with exactly $2n-2$ points in convex position in $\R^2$,  there exists a partition $\mathcal{P}'$ of $S'$ into $n-1$ color classes, where each color $i\in \{1,2,\dots n-1\}$ is used twice, such that
$T_{n-1}=\mathcal{N}(\mathcal{P}')$. Thus, we may assume that there exists a two-to-one ``coloring function" $\mathcal{C}:S' \to [n-1]$ that associates two points in $S'$ with a color $i$, (the color of node $v_i$). 

Let $S$ be a set of $2n$ points in convex position in $\R^2$, ordered in a clockwise manner, say $S=\{ \bs_1,\bs_2,\dots ,\bs_{2n}\}$, and assume without loss of generality that $\bs_1$ is at twelve o' clock. Next, consider the set $S':=S\setminus \{\bs_2,\bs_{2n}\}$. To this set $S'$ we can apply the
induction hypothesis, it is properly colored and gives $T_{n-1}$. Now we show how to add color $n$ to the remaining points in $S$ to give $T_n$. There are two cases.


\begin{itemize}
\item[Case 1]\label{Case1} If $\mathcal{C}(\bs_1)=r$, then extend $\mathcal{P}'$ to a partition $\mathcal{P}$ of $S$ by 
assigning color $n$ to the points $\bs_2$ and $\bs_{2n}$. Thus $\mathcal{P}=\mathcal{P}' \cup \{\bs_2,\bs_{2n}\}$. 
Let $L_n$ be the line through $\bs_2$ and $\bs_{2n}$.
Observe that on one side of $L_n$, say $L_n^+$, there is only $\bs_1$. Then the other points in $S'$ are contained in the other open half plane $L_n^-$. In particular, one point, say $\bs_j$, is such that $\mathcal{C}(\bs_j)=r$. Thus $\bs_1$ and $\bs_j$ have color $r$. Then $\conv(\bs_2,\bs_{2n})$ and $\conv(\bs_1,\bs_{j})$ intersect so $\mathcal{N}(\mathcal{P})$ contains the edge $(r,n)$. Furthermore, for any $i \neq n,r$, we have that $\mathcal{N}(\mathcal{P})$ does not contain the edge $(i,n)$, since the points with color $i$ are contained in $L_n^-$ and so their convex hull cannot intersect $\conv(\bs_2,\bs_{2n})$. Thus the nerve of $\mathcal{P}$ is $T_n$.
\smallskip

Before starting Case 2 consider the relabeling of  $S':=S\setminus \{\bs_2,\bs_{2n}\}= \{ \x_1=\bs_1,\x_2=\bs_3,\dots ,\x_{2n-2}=\bs_{2n-1}\}$.

\item[Case 2]\label{Case2}  
If $\mathcal{C}(\bs_1)\not=r$, then we know that on one side of the line $L_n$ (through $\bs_2$ and $\bs_{2n}$) there are two points in $S'$, say $\x_i, \x_{i+k}$ (as above) such that  $\mathcal{C}(\x_i)= \mathcal{C}(\x_{i+k})=r$ for $i\geq 3$ and $1< k\leq (2n-2)-i$.
Apply to $S'$ the following new coloring $\bar{\mathcal{C}}:S'\to [n-1]$ defined as  $\bar{\mathcal{C}}(\x_j)=\mathcal{C}(\x_{(j + 2n- i - 1)})$ mod($2n-2$). That is, the rotation that sends the corresponding color in $\x_i$ to $\x_1$. Observe that this rotation preserves all the intersection patterns that existed before (by Lemma \ref{ordertype}), and thus 
$\mathcal{N}(\mathcal{P}' )$  is $T_{n-1}$. Lastly, we are now in the position to apply Case 1
again, so the theorem follows.  \end{itemize} 

\begin{figure}[h]
\centering\includegraphics[scale=.35]{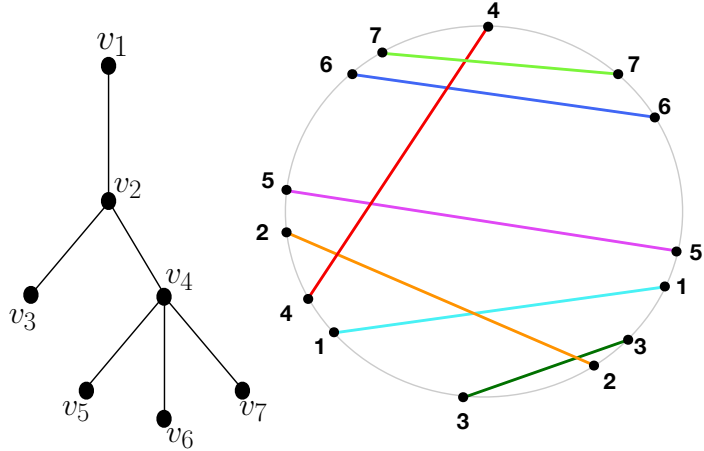}
\caption{Example of a tree with seven nodes and shown as partition induced on a set $S$ of $14$ points in convex position}\label{example1}
\end{figure}


This completes the proof that any set $S$ of $2n$ points in convex position in the plane have a partition whose nerve is isomorphic to any given tree $T_n$.
\end{proof}

To extend our result to the case that the points are in general position, we will use a  famous theorem in combinatorial geometry, the Erd\H{o}s-Szekeres Theorem. This theorem says that every sufficiently large set of points in general position contains a subset of $k$ points in convex position. The fact that this number $N=N(k,2)$ exists for every $k$ was first established in a seminal paper of Erd\H{o}s and Szekeres, \cite{erdosszekeres35} who proved the following bounds on $N(k,2)$. \[ 2^{k-2} + 1 \le N(k,2) \le { 2k-4 \choose k-2 } + 1 .\] A handful of recent papers have improved the upper bound (see for instance ~\cite{Morris} for an excellent survey and a very recent paper by A. Suk  \cite{Suk} showing that $N(k,2)=2^{k+o(k)}$).


By the Erd\H{o}s- Szekeres Theorem we know that ${4n-4 \choose 2n-2 } + 1$ points always contain a $2n$-gon. Then, we can use  Theorem \ref{treeconvex}. Finally we explain
how to extend the partition (or coloring) given by Theorem \ref{treeconvex} to the rest of the points in $\bar{S}$.

\begin{definition}\label{extendable} Let $S$ be a set of points in $\R^d$ and let $\mathcal{P} = S_1, \dots, S_n$ be an $n$-partition of $S$ into $n$ color classes 
that yields a specific nerve $\mathcal{N}(\mathcal{P})$.  We say that a $\mathcal{P}$ is \emph{extendable} if for all $\bar{S}$ containing $S$, there is a partition $\bar{\mathcal{P}} = \bar{S}_1 \dots \bar{S}_n$ of $\bar{S}$ extending $\mathcal{P}$ (meaning $S_i\subset \bar{S}_i$ for all $i$) such that $\mathcal{N}(\bar{\mathcal{P}})$ isomorphic to $\mathcal{N}(\mathcal{P})$.
\end{definition}

Observe that in general, such an extension is not necessarily possible, for example, Figure \ref{fignonextendable} shows a set of six vertices, and a partition in three color classes (see left side of the figure), that is not extendable. Note that any extension that includes the midpoint will change the intersection pattern (see right side of the figure). Surprisingly, in the case of the nerves of the partitions obtained in Theorem \ref{treeconvex} and (Theorem \ref{treescyclicD} in the next subsection), this extension is possible.

\begin{figure}[h]
\centering\includegraphics[scale=.50]{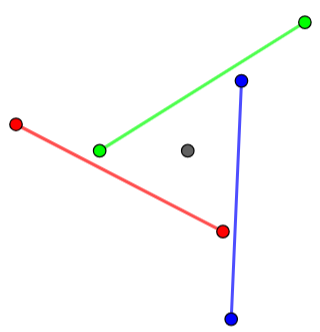}
\hskip .5cm \includegraphics[scale=.50]{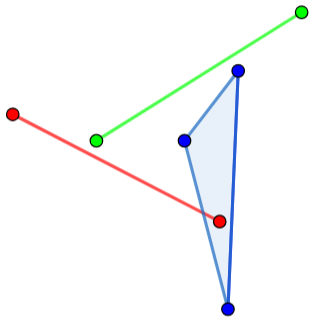}
\caption{}\label{fignonextendable}
\end{figure}


\begin{lemma}\label{extension2d}
Let $T_n$ be a given tree on $n$ nodes and let $S$ be a set of $2n$ points in convex position in the plane. Then the partition $\mathcal{P}$ obtained in proof of Theorem \ref{treeconvex} 
is extendable.
\end{lemma}

\begin{proof}
Let $\bar{S}$ be an arbitrary finite set of points in $\R^d$, such that $S\subset \bar{S}$. 
We begin by assuming that  the  ``color partition function"  $\mathcal{C}:S\to [n]$ is the 
one given in Theorem  \ref{treeconvex}. It 
yields a partition $\mathcal{P}$ of $S$ with nerve $\mathcal{N}(\mathcal{P})$ isomorphic to $T_n$ and $n$ is the last color added.
Recall that we denoted by $v_r$ the node in $T_{n-1}$ such that $\{v_n,v_r\}$ is the leaf of $T_n$ in which we added $v_n$. 

The extension of $\mathcal{P}$ of $S$ will be given through induction on $n$, by a ``color partition function" $\bar{\mathcal{C}_n}:\bar{S}\to [n]$ as follows:

\noindent a) For $n=1$, let $\bar{\mathcal{C}_1}(\x)=1$ for every point in $\bar{S}$. \\
\noindent b) For the induction step, the extension $\bar{\mathcal{C}}_{n-1}: \bar{{S}} \to [n-1]$ exists by induction hypothesis. Here is how we obtain the extension $\bar{\mathcal{C}}_{n}$:

Let $S_j$ denote the set of points in $S$ of color $v_j$, or $j$ for simplicity.
Consider the line $L_{n}$ through $S_n$ (it is given by points $\s_2$ and $\s_{2n}$ in Theorem \ref{treeconvex}), and recall that this line
leaves only one element of $S_{r}$ in one side, say $L_n^{+}$, and the rest of the points of $S$ in the other side $L_n^{-}$.  
We define $\bar{\mathcal{C}}:\bar{S}\to [n]$ as follows:  $ \bar{\mathcal{C}_n}(\x) = \bar{\mathcal{C}}_{n-1}(\x) \text{\ when \ } \x\in L_n^{-}$, \ \ $\bar{\mathcal{C}_n}(\x)= r \text{\ when \ } \x \in \conv(S_r)$, and, finally,
$\bar{\mathcal{C}_n}(\x)= n  \text{\ when \ } \x\in \cl(L_n)^{+} \text{\ but \ } \x\notin \conv(S_r).$ Here $\cl(L_n)^{+}$ denotes the closed half-plane at the right of $L_n$.

\begin{figure}[h]
\centering\includegraphics[width=5.42cm]{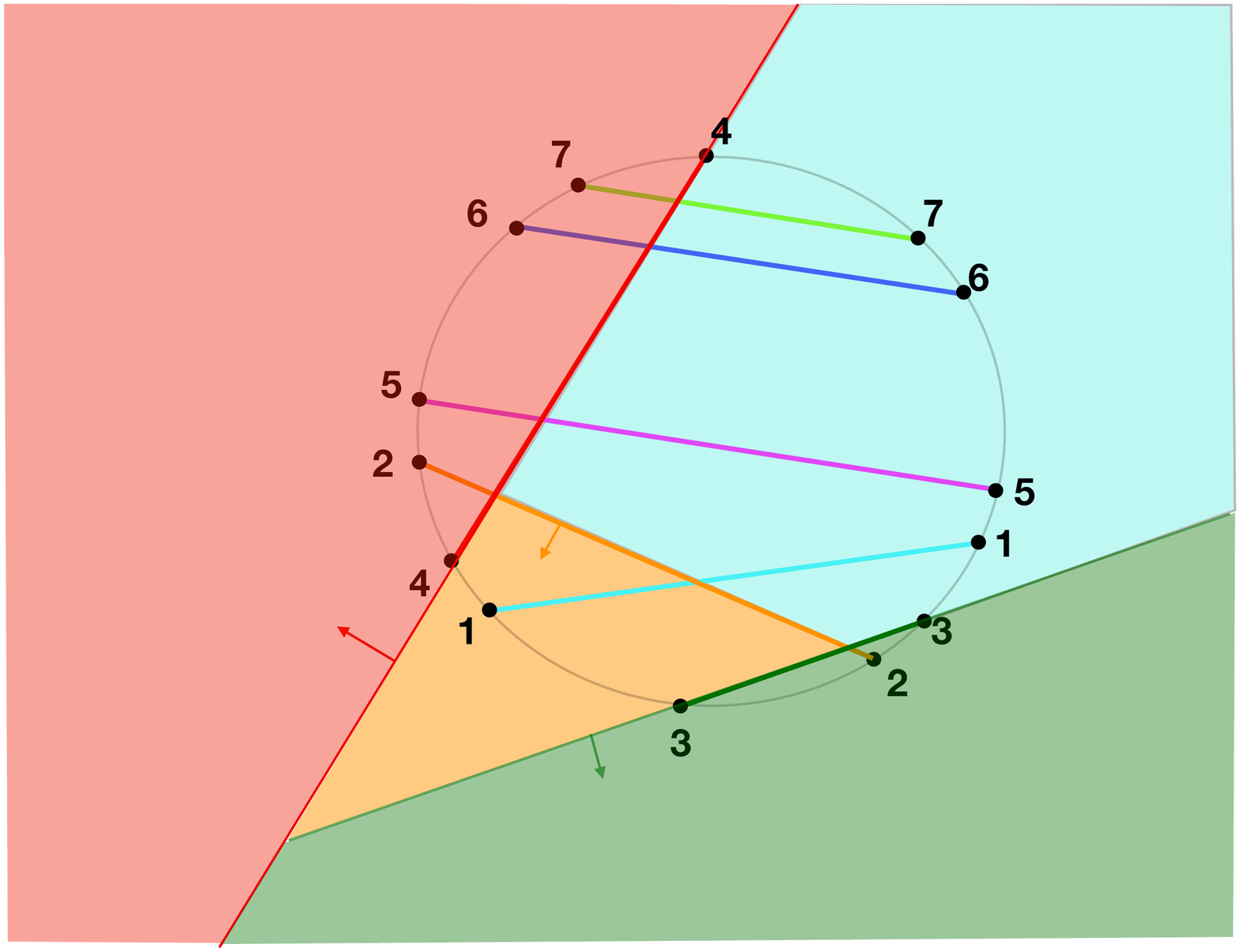} \includegraphics[width=5.42cm]{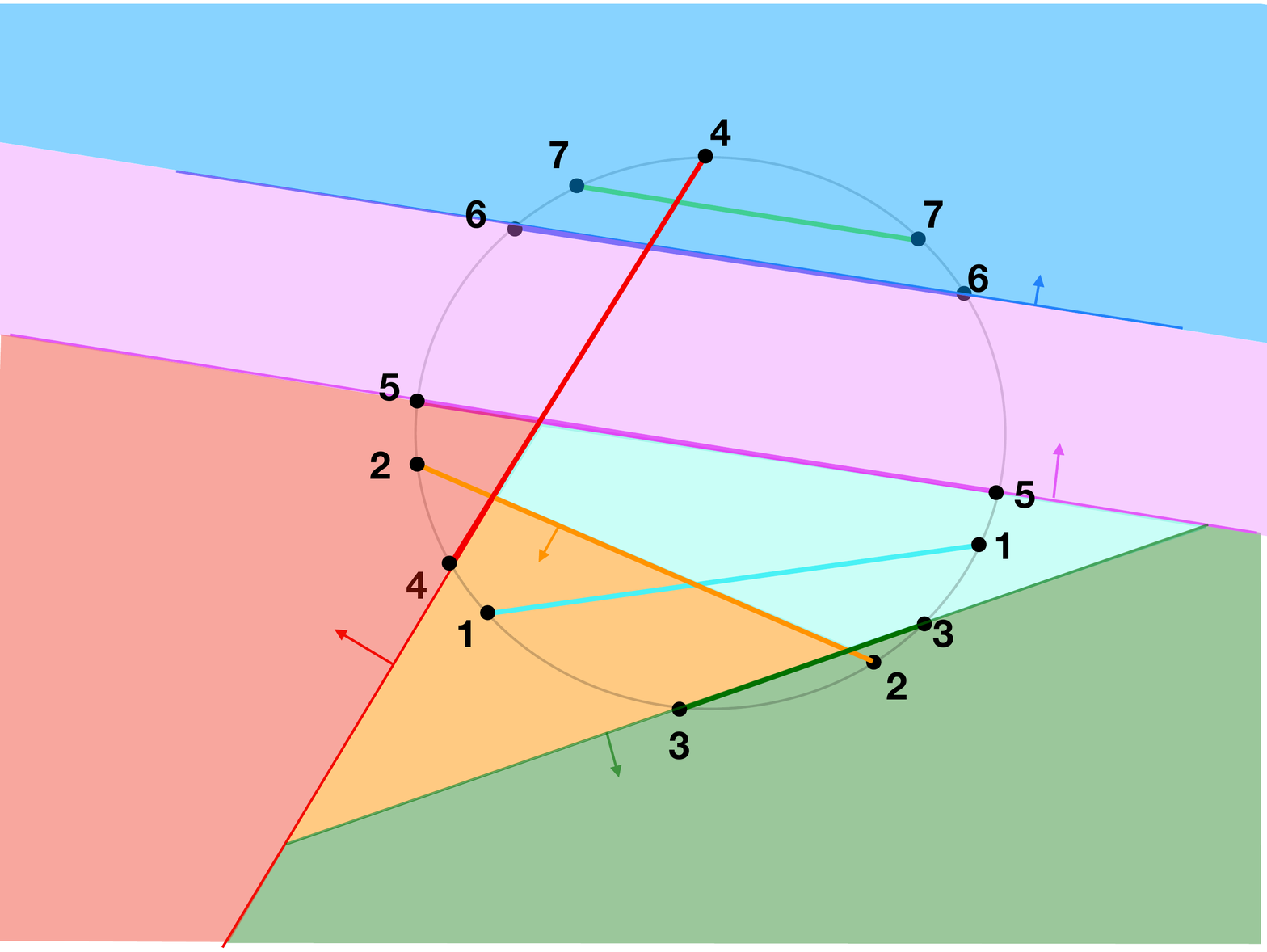} \includegraphics[width=5.42cm]{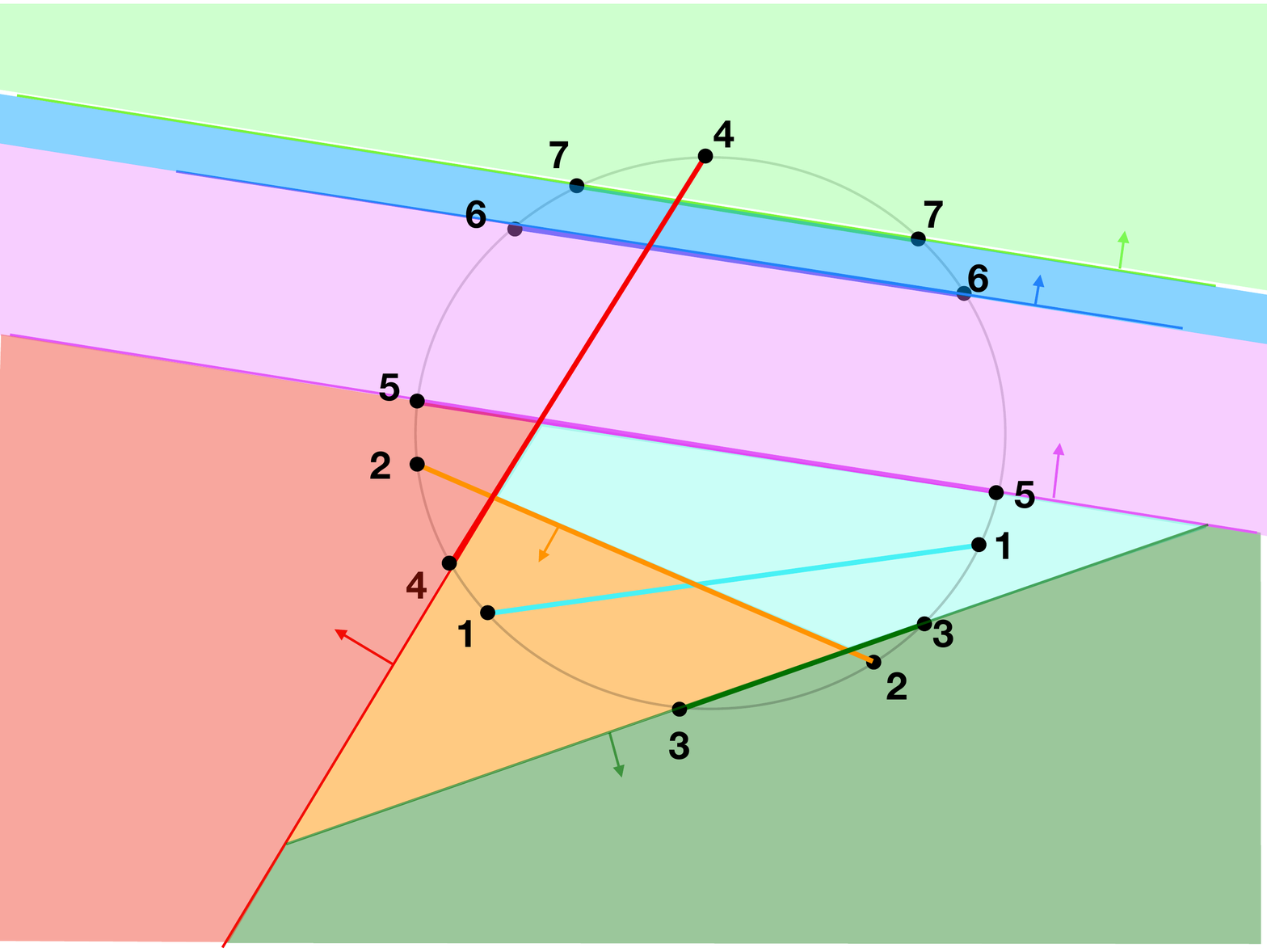}\\
\caption{The extension of the partition obtained in Figure \ref{example1}. The left figure, is the extension up to $n=4$, the central figure is the extension up to $n=6$, and
the right figure is up to $n=7$}\label{exampleextension}
\end{figure}

Observe that, by the induction process, the intersection pattern of 
$\bar{S}_1,\dots \bar{S}_{n-1}$ are the same in ${L_n}^{-}$ by construction. Furthermore, 
$\cl(L_n)^{+}$ does not intersect any other element in the partition, so no new 
intersections occur.
\end{proof}


\subsection{Proof of Theorem ~\ref{tverbertrees+cycles} (A) in $\R^d$}\label{treesind}

Next, we will show a general dimension version of Theorem ~\ref{treeconvex}. The
pattern of the proof is very similar to the planar case, but we will need to use properties
of cyclic polytopes and their oriented matroids.
 A  parametrized curve $\alpha : \R \longrightarrow \R^d$ is a  \emph{$d$-order curve} (sometimes called \emph{alternating}) when
no affine hyperplane $H$ in $\R^d$ meets the curve in more than $d$ points.
An example is the famous \emph{moment curve}. See  \cite{sturmfelsorder}, \cite{CordovilDuchet}, \cite{OMbook}.

In what follows we will use \emph{ordered cyclic $d$-polytopes} $C_m(d)$ which are obtained as the convex hull of $m$ 
vertices $S:=\{\x_1,\x_2,\dots \x_m\}$ along a $d$-order curve in $\R^d$ and thus, we may order the vertices of this polytope in an increasing sequential manner, 
say $\alpha(t_1)=\x_1<\alpha(t_2)=\x_2<\dots <\alpha(t_m)=\x_m$.  Ordered cyclic polytopes are very special because 
every subpolytope is also cyclic with respect to the same vertex order, i.e., the corresponding oriented matroid is  \emph{alternating}. 
Alternating means the chirotope has all positive signs. See Section 9.4 in the book \cite{OMbook}. 

\smallskip

\begin{theorem}\label{treescyclicD}
Let $T_n$ be any tree with $n$ nodes, and let $S$ be the vertices of an ordered cyclic $d$-polytope $C_m(d)$ with $m=(n-1)(d+1)+1$ vertices in $\R^d$. Then, there exists a partition 
$\mathcal{P}$ of $S$ such that the nerve $\N(\mathcal{P})$ is isomorphic to $T_n$. 
\end{theorem}

\begin{proof}
Let $C_m(d)$ be an ordered cyclic $d$-polytope, with $m$ vertices and assume as before $S:=\{\x_1,\x_2,\dots \x_m\}$ along the curve.
As in the case of the plane, the proof will be given by induction on $n$ the number of nodes of the tree $T_n$. 

If $n=1$, again there is nothing to prove. If $n=2$, the only tree with two vertices is $K_2$. Then by Radon's theorem, any set of $d+2$ points in $S$ 
can be partitioned in $S=S_1\cup S_2$ with $2\leq |S_i|\leq d$ for $i\in \{1,2\}$ and intersection graph $K_2$, see Figure \ref{example} on the left. 

\begin{figure}[h]
\centering\includegraphics[scale=.35]{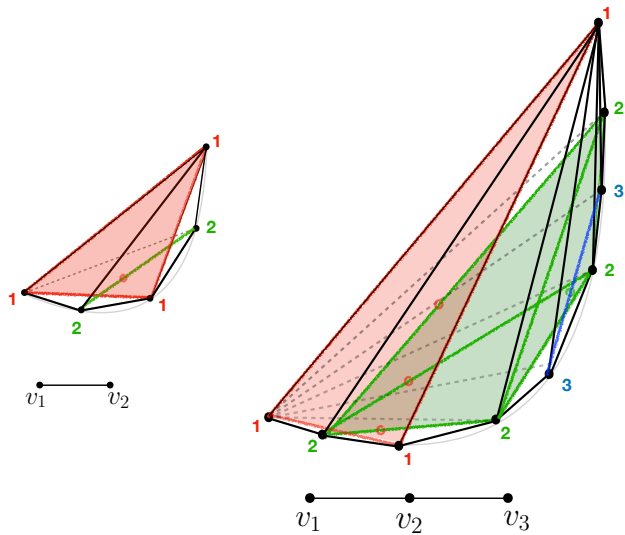}
\caption{Two examples, on the left, we show a tree on two nodes shown as a partition in a set $S$ of five vertices of the cyclic polytope $C_5(3)$. On the right side we represent a  tree on three nodes as a partition of the nine vertices of another cyclic polytope in $\R^3$, this time $C_9(3)$.}\label{example}
\end{figure}

For the induction step, suppose $T_n$ was obtained from $T_{n-1}$ by adding the node $v_n$ to a node $v_r\in T_{n-1}$ such that $\{v_n,v_r\}$ is a leaf of $T_n$, and assume that $T_{n-1}$ is the nerve of some set $\mathcal{N}(\mathcal{P}')$ where the set $S'$ are the vertices of the ordered cyclic polytope with exactly $(n-1)(d+1)-d$ vertices in $\R^d$ and $\mathcal{P}'=\{S'_1,S'_2,\dots ,S'_{n-1}\}$ are the color classes with color $1,2,\dots n-1$ respectively, via a ``coloring function" 
$\mathcal{C}':S'\to [n-1]$. 

Let $k$ the maximum number in $[n]$ such that $\x_k$ is in $S'_r$.  Next in $C_m(d)$ consider the subpolytope $Q$  of $(n-1)(d+1)-d$ vertices, obtained as the convex hull  $\conv(\x_1,\x_2,\dots ,\x_k , \x_{k+d+2},\dots \x_m)$, and $R$ is the polytope consisting of the convex hull of the complement of $Q$ and $\x_k$, thus $R=\conv(\{\x_k, \x_{k+1}, \dots, \x_{k+d+1}\})$. Note both $Q$ and $R$ are ordered cyclic polytopes and $Q\cap R=\{\x_k\}$. Thus by the induction hypothesis there exists a partition of the vertices of $Q$ into $n-1$ color classes whose nerve is isomorphic to $T_{n-1}$ as before.  Next, by Radon's Lemma there exists a partition into two color classes $A$ and $B$ of the $d+2$ vertices of $R$. 

Say $\x_k\in A$, then define a ``coloring function" $\mathcal{C}:S \to [n]$ in the following way: $\mathcal{C}(\x)=\mathcal{C}'(\x)$ if $\x$ is a vertex of $Q$,  $\mathcal{C}(\x)=r$ if $\x\in A$, and, finally, $\mathcal{C}(\x)=n$ if $\x\in B$. That is  $S_n\cap S_r \not=\emptyset$ and no further intersections occur. By the construction the parts of $\mathcal{P}$ consist of the $n$ color classes determined by
the coloring  $\mathcal{C}$. The nerve $\mathcal{N}(\mathcal{P})$ is isomorphic to $T_n$.
\end{proof}



In dimension two, we relied on Erd\H os-Szekeres to build a convex polygon.
For the general case in $\R^d$, we need a multi-dimensional version of Erd\"os-Szekeres theorem that follows from an application of the hypergraph Ramsey theorem~\cite{CDFhypergraphramsey}. The theorem we
need was first given by Gr\"unbaum (Exercise 7.3.6 in \cite{Gbook}) and Cordovil and Duchet \cite{CordovilDuchet} using oriented matroid methods. See Proposition 9.4.7 of \cite{OMbook} for a short proof. The theorem 
shows the existence of a number $N=N(k,d)$  such that every set of $N$ points in general position in $\R^d$ contains the vertices of an ordered cyclic $d$-polytope. 
 $N$ is bounded from above by the hypergraph Ramsey number $R_{d+1}(m)$ (see the Introduction)
 ensuring the existence of an alternating oriented matroid (hence an ordered cyclic polytope with $m$ vertices). 
  
According to \cite{sturmfelsorder}
when the oriented matroid is alternating, then its cyclic $d$-polytope is on a $d$-order curve in $\R^d$ and every subpolytope of it is also cyclic. 
This is quite a useful fortuity, since it is well known, that in odd dimensions there exist combinatorial cyclic polytopes with that some subpolytopes which are not cyclic (see page 104 of the same paper).
By these facts, we know that if $\bar{S}$ is a set of points in general position in $\R^d$ with at least  $R_{d+1}((n-1)(d+1)+1)$ points, then $\bar{S}$ contains a set $S$ consisting of the $m=(n-1)(d+1)+1$ vertices of an ordered cyclic $d$-polytope $C_m(d)$.

To finish the proof we just need to ``extend", as we did in the case of the plane, the partition given in Theorem \ref{treescyclicD} (for the vertices of $C_m(d)$) to a partition $\bar{P}$ of $\bar{S}$  in such a way that the nerve $\mathcal{N}(\bar{\mathcal{P}})$ is preserved. Lemma \ref{extension} below guarantees that this is always possible, finishing the proof of Theorem
\ref{tverbertrees+cycles} (A).



\begin{lemma}\label{extension}
Let $T_n$ be a given tree and let $S$ be the vertices of an ordered cyclic polytope with $m=(n-1)(d+1)+1$ vertices in $\R^d$.
Then the specific partition $\mathcal{P}$ obtained in Theorem \ref{treescyclicD} is extendable.
\end{lemma}

\begin{proof}

Let $\bar{S}$ be an arbitrary finite set of points in $\R^d$, such that $S\subset \bar{S}$. 
Let $S_j$ denote the set of points in $S$ of color $v_j$, or $j$ for simplicity.
Let us begin by assuming that  the  ``color partition function"  $\mathcal{C}:S\to [n]$, given in Theorem \ref{treescyclicD},
yields a partition $\mathcal{P}$ of $S$ with nerve $\mathcal{N}(\mathcal{P})$ isomorphic to $T_n$. The extension of $\mathcal{P}$ of $S$ will be given by induction on the number of nodes $n$. \\

\noindent a) In the case $n=1$ assign $\bar{\mathcal{C}_1}(\x)=1$ for every point in $\bar{S}$.\\
\noindent b) For the induction step note that the induction hypothesis guarantees 
the extension 
 $\bar{\mathcal{C}}_{n-1}: \bar{{S}} \to [n-1]$ exists. 
 

To begin observe that polytopes $Q$ and $R$ defined in Theorem \ref{treescyclicD} satisfy that 
$Q\cap R=\{\x_k\}$ so $(Q\setminus \{\x_k\})\cap(R \setminus \{\x_k \})=\emptyset$. 
Therefore, there exists a $(d-1)$-hyperplane $H$ that separates these two sets and 
leaves points of color $r$ in both sides of the hyperplane. Furthermore, $R$ is completely contained in the closure of one of the sides of this hyperplane, say $H_n^-$. 
The ``color partition function" $\bar{\mathcal{C}_n}:\bar{S}\to [n]$ is given as follows: \\

$ \bar{\mathcal{C}_n}(\x) = \bar{\mathcal{C}}_{n-1}(\x) \text{\ if \ } \x\in H_n^{-}$,\ \ $\bar{\mathcal{C}_n}(\x) = r \text{\ if \ } \x \in S_r$, and 
$\bar{\mathcal{C}_n}(\x) = n  \text{\ if \ } \x\in \cl(H_n^{+}) \text{\ and \ } \x\notin S_n$.

As before $\cl(H_n^{+})$ is the closed half-hyperplane containing only points in $R$ 
of color $n$ and color $r$.

Observe that by the induction process the intersection pattern of 
$\bar{S}_1,\dots \bar{S}_{n-1}$ are the same in ${H_n}^{-}$ by construction, and 
$\cl (H_n^{+}) \cap S_r\not= \emptyset$ yields the leaf $\{v_r,v_n\}$. Furthermore 
$S_n \subset \cl(H_n^{+}) $ does not intersect any other elements in the partition since they are
contained in $H_n^-$, so no further  intersections occur.

\end{proof}

\subsection{Proof of Theorem~\ref{tverbertrees+cycles} (B)}

\begin{proof}[Proof of Theorem \ref{tverbertrees+cycles} part (B)]
Suppose that $\bar{S}$ is a set of at least $nd + n + 4d$ points in general position in $\R^d$. We start by projecting the points onto a generic $2$-plane $H$ where we can assume, without loss of generality, that the points of $\bar{S}$ have distinct projections  onto it.  Let $S$ be the 
projection of $\bar{S}$, now planar points.

\begin{lemma}~\label{pizza}
There exists a circle $C$ containing all these projected planar points in $S$, and a subdivision $C'$ of $C$ into $n$ sectors such that:\\
(i) Each sector contains at least $d+1$ points.\\
(ii) No two adjacent sectors form a combined angle of more than $\pi$ radians.
\end{lemma}

\begin{proof} We start by picking a line $L_1$ with at least $\lfloor \frac{nd + n + 4d}{2}\rfloor$ points on both sides of $L_1$. 
Denote by $L_{1}^-$ and $L_{1}^+$ respectively, the open half-spaces defined by $L_1$ and $M_1^+, M_1^-$ the points of $S$ on the two half-spaces of $L_1$. Applying the Ham Sandwich Theorem (see Section 1.3 in \cite{Mbook}) to the sets $M_{1}^-$ and $M_{1}^+$, we can find a line $L_2$ 
so that  $L_1$ and $L_2$ together separate the plane into four regions, say $R_1$, $R_2$, $R_3$ and $R_4$ with at least$\lfloor \frac{nd + n + 4d}{4}\rfloor$ projected points in each region. Note that  $\lfloor \frac{nd + n + 4d}{4}\rfloor \geq d+1$ points.

Denote by $\p$ the point in the plane where $L_1$ and $L_2$ intersect, and let $C$ be a circle centered at $\p$ that contains all the projected points. Now we choose arcs emanating from $\p$ to subdivide each of the four regions $R_1$,$R_2$,$R_3$, and $R_4$ into as many subregions, containing at least $d+1$ points (note that each of the $R_i$ have at least $d+1$ points in them by construction).
This can be done as follows: 

If $R_{i}'$ has at least $2d+2$ points, then take a line emanating from $\p$ and rotate it until it divides $R_i$ into two regions, one with $d+1$ points denoted $R_{i1}$, and the other with the remaining (at least $d+1$) points in $R_i$ denoted $R_{i}'$. Otherwise do nothing.
Repeating this process as many times as possible, we will obtain a subdivision of each $R_i$ into subregions, all but one of which have 
exactly $d+1$ points, and none of which have more than $2d+1$ points. We call the final regions of this recursive process sectors.

%
Since the original four regions  $R_1, R_2, R_3, R_4$ satisfy property claim (ii) of the lemma, and process of subdivision is
made to show claim (i) holds after subdividing the four regions, all we have left to do is to check there are $n$ sectors. For this,
let $k_1, k_2, k_3,$ and $k_4$ denote the respective number of  sectors formed from each of the four regions, and $j_1, j_2, j_3$ and $j_4$ denote the number of points in each region. It suffices to show that $k_1 + k_2 + k_3 + k_4 \geq n$ because we can always merge adjacent
sectors within the same region $R_i$, while preserving claims (i) and (ii).

Our procedure for generating subdivisions guarantees that 
$j_i \leq k_i(d+1) + d$ for all $i= 1,2,3,4$. Summing these inequalities we get 
$j_1 + j_2 + j_3 + j_4 \leq (k_1 + k_2 + k_3 + k_4) (d+1) + 4d,$ so 
$$nd + n + 4d \leq (k_1 + k_2 + k_3 + k_4) (d+1) + 4d,$$ 
which implies that $(k_1 + k_2 + k_3 + k_4) \geq n$. This completes the proof of the lemma.
\end{proof}
Now we will use the the subdivision $C'$, whose existence is guaranteed by Lemma~\ref{pizza}, to find our 
desired partition of the data points whose partition nerve is an $n$-cycle.

We construct a partition one sector at a time. In the first step, we notice that one of the $n$ sectors, say $Q_1$, has at 
least $d+2$ points by the pigeonhole principle.\\
Use Radon's Lemma to partition the points in $Q_1$ into two sets $S_1$ and $S_2$ so that the convex hulls of $S_1$ and $S_2$ intersect.\\

\begin{center}\includegraphics[width=5cm]{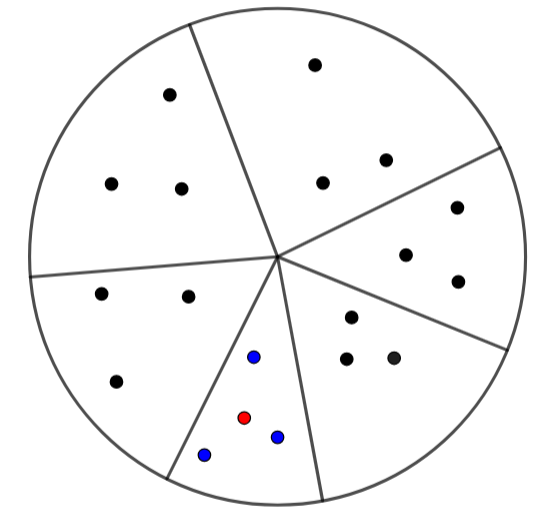}\includegraphics[width=5cm]{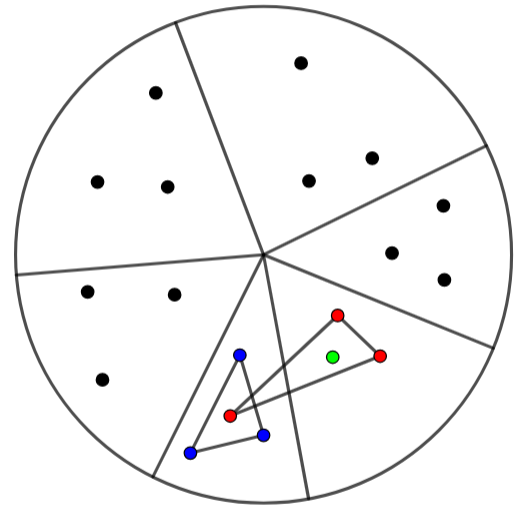}\includegraphics[width=5cm]{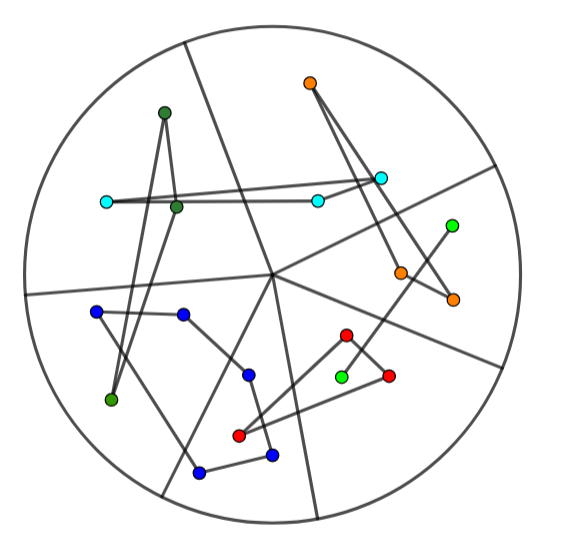}\\

From left to right: Illustration of the construction at steps 1 and 2, and the resulting partition.  \end{center}

In the second step, we denote the slice counterclockwise to $Q_1$ as $Q_2$.\\ By Radon's Lemma, any 
point $\x_2$ in $S_2$ from step 1, combined with the (at least) $d+1$ points in $Q_2$ can be partitioned into two sets 
$S_2'$ and $S_3$ so that the convex hulls of $S_2$ and $S_3$ intersect. Without loss of generality we can assume 
that $\x_2 \in S_2'$, and then set $S_2 = S_2 \cup S_2'$.  In step $k$, where $3 \leq k \leq n-1$, we continue in the 
same way. We denote the slice counterclockwise to $Q_{k-1}$ as $Q_k$.\\ By Radon's Lemma, any point $\x_{k}$ in 
$S_{k}$ from step $(k-1)$, combined with the (at least) $d+1$ points in $Q_{k+1}$ can be partitioned into two sets 
$S_{k}'$ and $S_{k+1}$ so that the convex hulls of $S_{k}'$ and $S_{k+1}$ intersect. Without loss of generality we can 
assume that $\x_{k-1} \in S_{k-1}'$, and then set $S_{k} = S_{k} \cup S_{k}'$.  Finally, in step $n-1$ we set $S_1 = S_1 \cup S_n$.

We claim that the nerve of the resulting partition $\mathcal{P} = \{S_1, S_2, \dots, S_n\}$ is the $n$-cycle.\\
This is a consequence of two facts: We used Radon's Lemma to guarantee that any two subsets appearing in the same 
sector have intersecting convex hulls. Each subset appears in at most two sectors, and since two adjacent sectors have a 
combined angle of at  most $\pi$ radians there is a line separating any two subsets that do not appear in the same sector.
Thus we have that $\conv(S_i) \cap \conv(S_j) \neq \emptyset$ if and only if there is some sector containing points from both 
$S_i$ and $S_j$. If we let ${\bf v}_i$ denote the vertex of the nerve corresponding to subset $S_i$, we see that the edges of 
$\mathcal{N}(\mathcal{P})$ consist precisely of $({\bf v}_n, \bf{v}_1)$ and $(\bf{v}_i, \bf{v}_{i+1})$ where $i \in [n-1]$. This completes the proof.
\end{proof}


\section{Improved Tverberg Numbers of Special Trees and in Low Dimensions}~\label{specialtrees}

\subsection{Proof of Theorem \ref{improvedtrees}: Better bounds for Tverberg numbers of caterpillar trees}

To make the notation easier, we adopt the following convention throughout the proof of Theorem \ref{improvedtrees}: All point sets $S \subset \R^d$ are indexed in increasing order with respect to their first coordinate. That is, if $S = {\x_1, \x_2, \dots, \x_n}$,
with $\x_i = (x_{i1}, x_{i2}, \dots, x_{id})$, then we assume that $x_{11} \leq x_{21} \leq \dots \leq x_{n1}$. Furthermore, by rotating the axes, we can assume that no two points have the same first coordinate and that the previous inequalities are strict.

We first prove the special case of stars in Theorem \ref{improvedtrees} as a lemma. A caterpillar is a sequence of stars, thus we can later use induction again.

\begin{lemma} \label{juststar}
For any $(d+1)(n-1)+1$ points in $\R^d$, we can find a partition of those points with nerve $St_n$, the star tree  on $n$ vertices (i.e., with $(n-1)$ spokes).
\end{lemma} 

\begin{proof}[Proof of Lemma \ref{juststar}]
We prove this by induction on $n$. For $n = 1$, the partition of one point to get $St_1$ is obvious.
	Now assume the result is true for some $n$. We need to show that any $(d+1)n + 1$ points can be partitioned with partition nerve $St_{n+1}$. Let $M = (n-1)(d+1)+1$. By induction hypothesis, the subset $\{\x_1, \dots, \x_M\}\subset S$ admits a partition $\mathcal{P} = \{A_1,\ldots,A_n\}$ with $\mathcal{N}(\mathcal{P}) \simeq St_n$.  Without loss of generality, assume that $A_1$ is the central vertex of the star graph. Let $\x \in S$ be some point in $A_1$. By Radon's lemma, there is a way to partition the $d+2$ points $\x, \x_{M}, \x_{M+1}, \dots \x_{M+d+1}$ into two sets $X,Y$ with $\conv(X) \cap \conv(Y) \neq \emptyset$, and we can assume that $\x \in X$. The set $Y$ intersects $A_1 \cup X$ but does not intersect any of $A_i$, $2 \le i \le n$, because every point in $Y$ has larger first coordinate than any point in $A_i$. Then we see $\{A_1\cup X, A_2,\ldots, A_n, Y\}$ is a partition which will induce the star graph $St_n$. 
\end{proof}

	
\begin{proof}[Proof of Theorem \ref{improvedtrees}]	
Now we prove that  for every caterpillar tree $T_n$ with at most $n$ nodes, every set $S$ with 
at least $(d+1)(n-1) +1$ points in $\R^d$ admits a partition $\mathcal{P}$ with $\mathcal{N}(\mathcal{P}) \simeq T_n$. 
An illustration of the partition constructed in the proof is given in the Figure~\ref{caterpillars}.
The proof is by induction on the length of the central path in $T_n$, which we will denote by $m$. The induction hypothesis says that
for every $m \in \mathbb{N}$ and any caterpillar tree $T_n$ with $n$ vertices and a central path of length $m$ the following two statements hold:\\
(1) Every set $S$ of $(d+1)(n-1) +1$ points in $\R^d$ admits a partition $\mathcal{P}$ with $\mathcal{N}(\mathcal{P}) \simeq T_n$\\
(2) Denote by $v$ the last vertex of the central path, and denote by $St_{k+1}$ the star subgraph induced by $v$ and its $k$ neighbors. Then the subsets in $\mathcal{P}$ corresponding to vertices in $St_{k+1}$ are comprised of the $(d+1)k + d + 1$ points in $S$ with largest first coordinate.\\

If the length of the central path is one, both parts of the induction hypothesis follow by applying Lemma~\ref{juststar}.	
	Assume the result holds if the central path is of length $m$. We consider caterpillar graphs which have central paths of length $m+1$. Let $G$ be such a graph with $n$ vertices. 
	We consider the endpoint of the path $v_{m+1}$ and the vertex prior $v_{m}$. If we consider the subgraph of $G$ consisting of the path $v_1, \ldots, v_m$ and all vertices adjacent to it except $v_{m+1}$, 
	this is a caterpillar graph with a path of length $m$. Let $p$ denote the number of vertices of this graph. By induction hypothesis, we can represent this graph using the $(d+1)(p-1)+1$ points $\x_1, \dots, \x_{(d+1)(p-1)+1}$.
We will have the partition $\{A_1,\dots, A_p \}$ where we take $A_1$ to be the set corresponding to $v_m$. Then take a point $\x \in A_1$ and the next $d+1$ points $\x_{(d+1)(p_1)+2}, \dots, \x_{(d+1)(p_1)+d+2}$ to have a Radon partition $X, Y$ with $\x \in X$. Our new partition will be $\{A_1 \cup X, A_2, \ldots, A_p, Y\}$. $Y$ will correspond to the vertex $v_{m+1}$ and will not intersect any of the other sets due to having larger first coordinate. In addition, $A_1 \cup X$ will not intersect any new sets by how we have arranged the points due to the induction hypothesis. Now as in the proof of the lemma, we can add new sets by considering $(d+1)$ points in iteration for each of the other vertices adjacent to $v_{m+1}$. Since there were $n-p$ vertices and we used $(d+1)$ points for each, in total we used $(d+1)(n-1)+1 + (d+1)(n-p) = (d+1)(n-1)+1$ points. This is the desired number.

Thus we have proved the induction hypothesis. To complete the proof of the theorem, we note that if we have more than $(d+1)(n-1)+1$ points, we can apply the induction hypothesis to find the desired partition of the $(d+1)(n-1)+1$ points $\x_1, \x_2, \dots \x_{(d+1)(n-1)+1}$, then add any remaining points to the subset corresponding to the endpoint of the central path in the caterpillar graph. 
\end{proof}

\begin{figure}[h]
		\centering\includegraphics[width=200px]{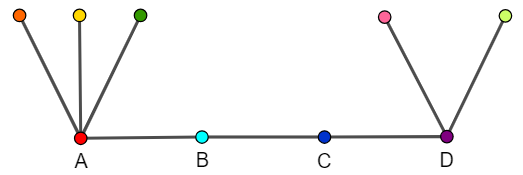}
\caption{An example caterpillar graph $G$ with 9 vertices.}
	\centering\includegraphics[width=440px]{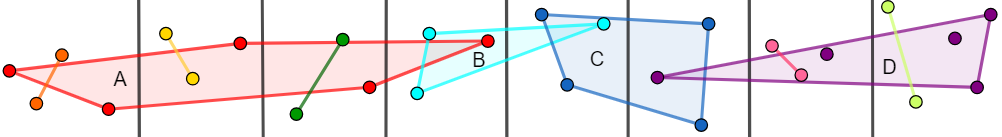}
		\caption{An example of how a set of points can be partitioned with nerve $G$. The vertical lines indicate how we start with a Radon partition of the leftmost $d+2$ points, then partition the points from left to right, considering $d+1$ more points at each step. Notice there are extra points on the right, which are added to the subset corresponding to the last vertex on the central path.}~\label{caterpillars}
		\label{star-proof1}
	\end{figure}

\subsection{Proof of Theorem \ref{lowdim}:  Tverberg numbers of trees in dimension two}

Now we focus on the situation in two dimensions. 

\begin{lemma}~\label{goodline}
	Let $S \in \mathbb{R}^2$ be a set of $2n$ points in the plane. Let $L_{\p_1\p_2}$ denote the line segment between points $\p_1$ and $\p_2$. Suppose that there exists $\p_1, \p_2 \in X$  such that  $L_{\p_1\p_2}$ divides the remaining points into two sets $A, B$ each of size $n-1$ and such that for any $\boldsymbol{a} \in A, \boldsymbol{b} \in B$, we have that $L_{\boldsymbol{a}\boldsymbol{b}}$ intersects $L_{\p_1 \p_2}$. Then it is possible to pair off elements $\boldsymbol{a}_i \in A$, $\boldsymbol{b}_i \in B$, such that for $i,j=1,\ldots,n-1$, $i \ne j$, $L_{\boldsymbol{a}_i\boldsymbol{b}_i}$ does not intersect $L_{\boldsymbol{a}_j\boldsymbol{b}_j}$. 
\end{lemma}

\begin{proof}
	Suppose we have points $\p_1$ and $\p_2$ as hypothesized and partition the remaining points into $A$ and $B$. Let $L$ be the line between $\p_1$ and $\p_2$. To pair off the points, we consider the vertices of $conv(A\cup B)$. Since $L$ separates the points of $A$ and $B$, we must have that there are a pair of adjacent vertices of $conv(A\cup B)$ such that one, $\boldsymbol{a}_1$, is a member of $A$ and the other $\boldsymbol{b}_1$, a member of $B$. The segment between this pair cannot intersect the segment between any other pair of points as this segment forms the boundary of the convex hull. We pair off these two points and then consider $conv(A\setminus\{\boldsymbol{a}_1\}\cup B\setminus\{\boldsymbol{b}_1\})$. We see that $L$ separates $A\setminus\{\boldsymbol{a}_1\}$ and $B\setminus\{\boldsymbol{b}_1\}$, so we can repeat this argument to pair off $\boldsymbol{a}_2$ and $\boldsymbol{b}_2$. Continuing in this fashion until we have paired off all the elements, we will have a pairing $(\boldsymbol{a}_1, \boldsymbol{b}_1), (\boldsymbol{a}_2, \boldsymbol{b}_2), \ldots, (\boldsymbol{a}_n, \boldsymbol{b}_n)$ where $L_{\boldsymbol{a}_i\boldsymbol{b}_i}$ does not intersect $L_{\boldsymbol{a}_j\boldsymbol{b}_j}$ for $i \ne j$.
\end{proof}

\begin{proof}[Proof of Theorem \ref{lowdim} part (A)]
	Let $A \in \mathbb{R}^2$ be a collection of $2n$ points in general position in the plane. Our goal will be to find a pair of points which can separate the remaining points into two sets of equal size so we can apply the above lemma. This will not always be possible, so we will try to make the size of the two sets as close as possible.
	
	To do this, we will consider the vertices of the convex hull of $A$. We pick arbitrarily a vertex $\p_1$ of $conv(A)$ and order the remaining vertices $\p_2,\ldots,\p_k$ in counter-clockwise order where $k$ is the number of vertices. For $i=2,\ldots,k$, we divide the remaining vertices of $A$ into two sets $B_i, C_i$ where $B_i$ is the set of vertices in $A$ to the left of $L_{\p_1\p_i}$ and $C_i$ is the set of vertices to the right of $L_{\p_1\p_i}$. We note that the size of $B_i$ decreases from $2n-2$ to $0$ as $i$ increases and the size of $C_i$ increases from $0$ to $2n-2$.
	
	We consider two cases. The first case is that there exists $i$ such that $|B_i| = |C_i| = n-1$ and then we can apply the above lemma as the line segment between every pair of points in $B_i \times C_i$ intersects $L_{\p_1\p_i}$ since $L_{\p_1\p_i}$ separates $B_i$ and $C_i$ and $\p_1,\p_i$ are vertices of $\conv(A)$. Then we have a pairing $(\boldsymbol{b}_1, \boldsymbol{c}_1), \ldots, (\boldsymbol{b}_{n-1}, \boldsymbol{c}_{n-1})$ where for any two pairs the segments do not intersect, but each intersects $L_{\p_1\p_i}$. Then the partition $\{\{\boldsymbol{b}_1, \boldsymbol{c}_1\}, \ldots, \{\boldsymbol{b}_{n-1},\boldsymbol{c}_{n-1}\}, \{\p_1, \p_i\}\}$ is a partition which induces the star graph $S_n$. For an example of this case and how to partition the points, see Figure \ref{star-proof-case1}.
	\begin{figure}[h]
		\centering\includegraphics[width=150px]{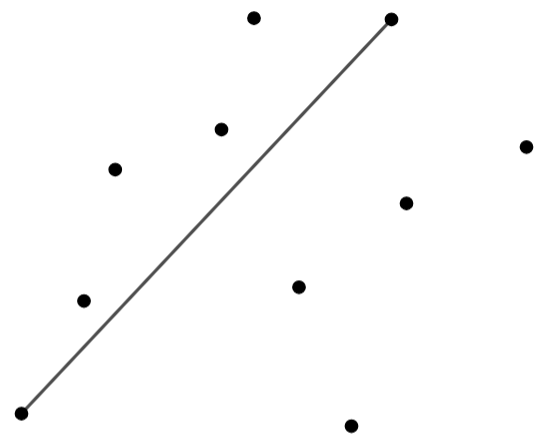}
		\centering\includegraphics[width=150px]{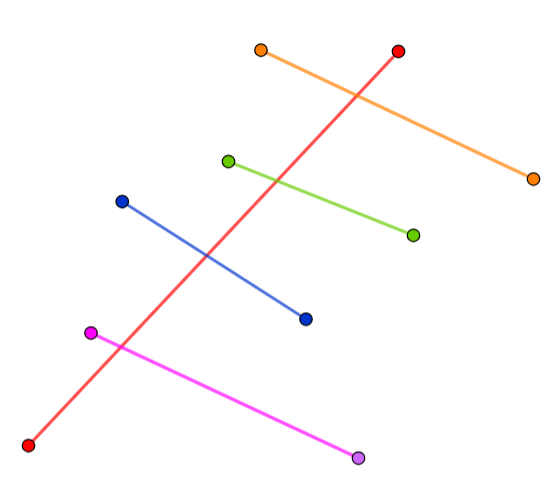}
		\caption{In the first case, there is a partition which divides the remaining points into two sets of equal size. Then we can pair off points such that the segment connecting them intersects the dividing line, but no other segment.}
		\label{star-proof-case1}
	\end{figure}
	
		The second case is that there does not exist such an $i$. In this case, we find $i$ such that $|B_i| > |C_i|$ and $|B_{i+1}| < |C_{i+1}|$.  Set $D = \{\p_1, \p_i, \p_{i+1}\}$ and notice that $\conv(D)$ must contain at least one point of $S$ in it's interior. $D$ will form the center vertex of our star graph.  See Figure \ref{star-proof1} for a depiction of this central triangle.
	
	\begin{figure}[h]
		\centering\includegraphics[width=100px]{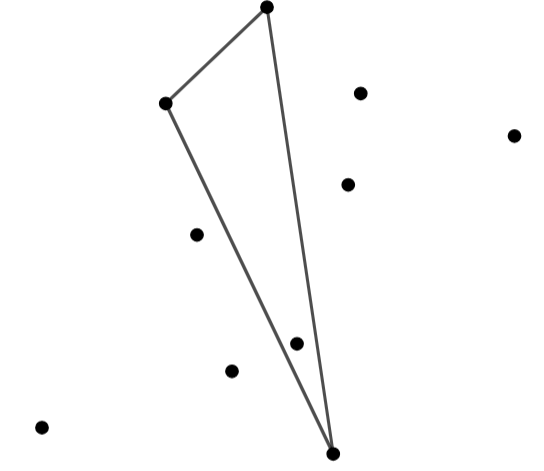}
		\centering\includegraphics[width=100px]{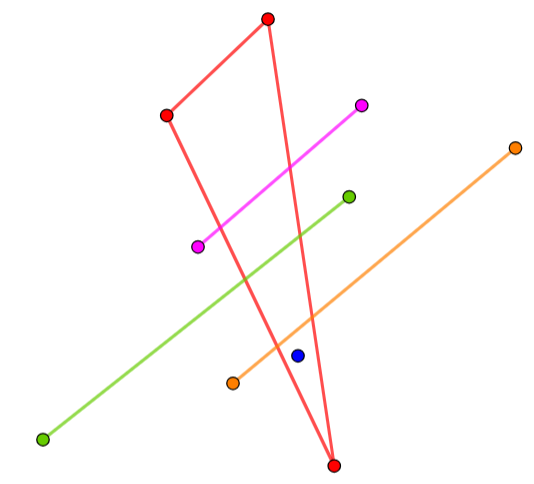}
		\caption{In the second case, we find a central dividing triangle of a given point configuration.  Then we pair off as many points on opposite sides of the triangle as possible using Lemma~\ref{goodline}, and make points in the interior of the triangle singletons until we have $n$ subsets. Any extra points are added to the subset containing the central dividing triangle.}
		\label{star-proof1}
	\end{figure}
	
 Now, using the above lemma, pair off points from $B_{i+1}$ and $C_i$ to form disjoint segments which will intersect $conv(D)$, and let every point in the interior of $D$ be a singleton (which will not intersect any of the segments since the points are in general position). Denote this partition as $\mathcal{P}$.
 
 $\mathcal{N}(\mathcal{P})$ is clearly a star graph, so it suffices to show $\mathcal{N}(\mathcal{P})$ has at least $n$ nodes (we can merge the subsets corresponding to any extra nodes with $D$, as $conv(D)$ already intersects every subset). To see this, first note that average number of points in each subset of $\mathcal{P}$ is at most two, since $\mathcal{P}$ has one subset of size three, at least one singleton, and the rest of the subsets are either singletons or pairs. On the other hand, the average number of points in each subset is equal to $2n$ divided by the number of subsets, so there must be at least $n$ subsets in $\mathcal{P}$. Thus $\mathcal{N}(\mathcal{P})$ has at least $n$ nodes, as was to be shown.
	
\end{proof}

Now we move to the proof of Theorem \ref{lowdim} part (B):
As a consequence of Lemma~\ref{ordertype}, when enumerating partition induced graphs it is enough to consider the partitions of combinatorial types of point sets. We can check whether a given simplex complex is $2$-partition induced on a representative for 
each order type.
	

To complete part (B) we relied on an explicit computer enumeration of all order types on small points set provided by  \cite{Aichholzer}. 
There exists exactly one point configuration for which it is impossible to generate $P_4$. This point configuration is displayed in 
Figure \ref{no4path}.  Its specific coordinatization is $A (222,243)$, $B (238, 13)$, $C (131,50)$, $D (154,105)$, $E (166,145)$ ,$F (134,106)$
$G (174,188)$, $H (18,51)$. 
For every other point configuration, we found a partition which induced the path graph $P_4$.
From this we assert that $\Tv(P_4,2) \ge 9$.  Since  we also found a partition inducing $P_4$ for every single order type on nine points
we are sure that $\Tv(P_4,2)=9$ because in the case $10$ or more points, we can use the weaker bound given in the proof of Theorem~\ref{improvedtrees} part (B).

Similarly for the cycle $C_4$ we have the configuration with coordinates $A (0,0)$, $B (8,5)$, $C (18,3)$, $D (7,4)$,
$E (14,5)$,
$F (10,8)$,
$G (11,7)$,
$H (14,17)$,
$I (11,6)$,
$J (12,12)$, which gives the desired lower bound. The upper bound is given by following the proof of   Theorem~\ref{tverbertrees+cycles} part (B), except starting with any set of 13 points (the bound given in the theorem is higher since it accounts for divisibility issues that can occur in certain cases).

\begin{figure}[h]
	\centering\includegraphics[width=200px]{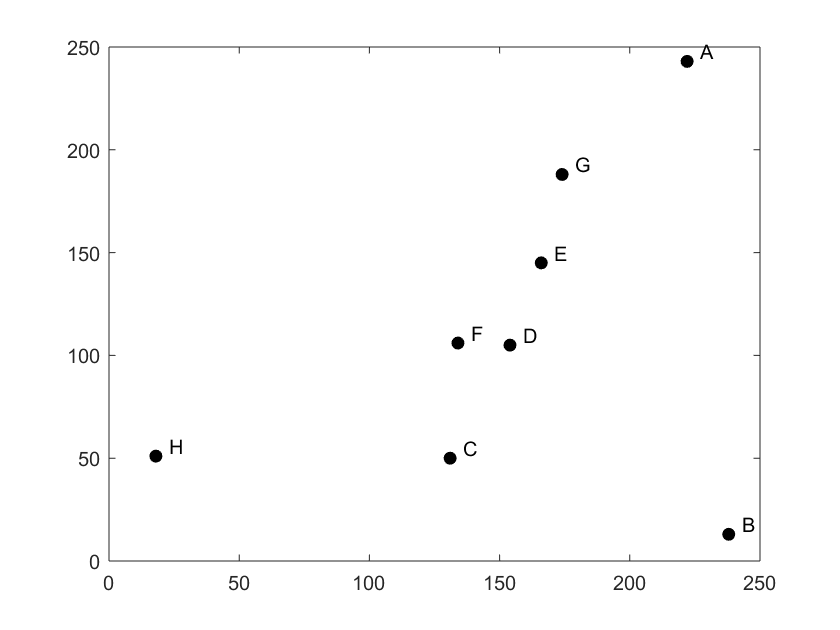} \hskip .5cm \centering\includegraphics[width=200px]{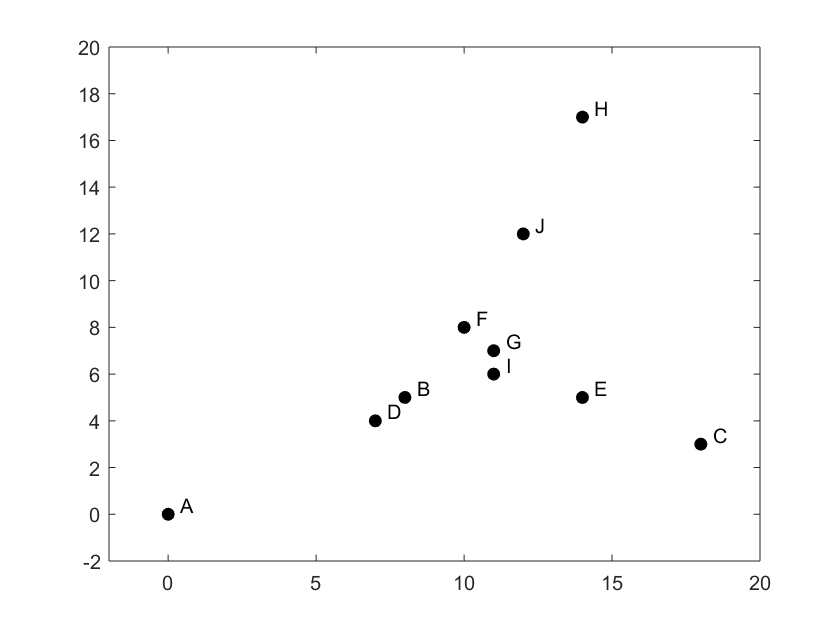}
	\caption{Two points configuration which cannot be partitioned to induce, respectively, $P_4$ (left on eight points) nor $C_4$ 
	(right ten points).} \label{no4path}
\end{figure}


\section{Final remarks and open problems} \label{final}

In this paper we generalize Tverberg's theorem by showing that many simplicial complexes, called Tverberg complexes, are always induced as the nerve of some partition of any sufficiently large set of points in a fixed dimension. But the study of Tverberg complexes abounds with unsolved questions. We conclude by listing a few:
\begin{enumerate}
\item What is the exact value of $\Tv(T_n,d)$ where $T_n$ is a tree with $n$ nodes?  Is $(d+1)(n-1)+1$ the correct value? What about the case of $d=2$?
\item What is the computational complexity of determining if a point configuration can partition induce a given graph?
\item What is the computational complexity of computing the Tverberg numbers of a given Tverberg complex, such as a tree?
\item Are there topological versions of Tverberg theorems for other simplicial complexes?
\item Is there a graph $G$ which is not $3$-Tverberg?
\item Is there a complex $K$ which is not $d$-Tverberg for any $d$?
\item Is there a complex $K$ and $i,j \in \mathbb{N}$, $i <j$ so that $K$ is $i$-Tverberg but not $j$-Tverberg? 
\end{enumerate}

{\bf Acknowledgements:} We are truly grateful to Florian Frick, Steve Klee,  Fr\'ed\'eric Meunier, Luis Montejano, David Rolnick, and 
Pablo Sober\'on, who gave us useful suggestions and encouragement in the early stages of this project. The research of the first, second and 
fourth author  were partially supported by NSF grant DMS-1522158 and NSF grant DMS-1818969. Deborah Oliveros was supported 
by PASPA (UNAM) and CONACYT  during her sabbatical visit to UC Davis, as well as by Proyecto PAPIIT 104915, 106318 and CONACYT Ciencia B\'asica 282280.  
This author would also like to express her appreciation of the UC Davis Mathematics department's hospitality during her visit.

\bibliographystyle{amsalpha}
\bibliography{references__1}

\section{Appendix: Proofs of auxiliary lemmas}\label{appendix}

In this appendix we include proofs of some supplementary lemmas mentioned in the introduction.

\begin{proof}[Proof of Lemma \ref{lowerb}]
Suppose by contradiction $\Tv(K,d)< 2n$. Let $S \subset \R^d$ be a set of points in convex position with $|S|=\Tv(K,d)$. 
By the pigeonhole principle, if we partition $S$ into $n$ disjoint subsets, there must be at least one subset that is a singleton $\{\x\}$. 
Since $K$ is connected, the node corresponding to the singleton $\{\x\}$ is connected, by an edge, to at least one other node, 
implying that $\{\x\}$ is in the convex hull of another subset. However, this is a contradiction as the points are in convex position. 
\end{proof}

\begin{proof}[Proof of Lemma \ref{ordertype}]
\sloppy
To show that $\mathcal{N}^1(\mathcal{P}) = \mathcal{N}^1(\sigma({\mathcal{P})})$ it suffices to show that 
$\conv(P_{i_1})\cap \conv(P_{i_2}) \neq \emptyset$ if and only if $\conv(\sigma(P_{i_1}))\cap \conv(\sigma(P_{i_2})) \neq \emptyset$ for all $i_1, i_2 \in [n]$.
Suppose $\conv(P_{i_1})\cap \conv(P_{i_2}) \neq \emptyset$. Then they contain respectively $P_{i_1}'$ and $P_{i_2}'$ which are an inclusion minimal Radon partition of $S_1$.
 Since $\sigma$ is an order-preserving bijection, $\sigma$ is an isomorphism between oriented matroids (see, for instance \cite{ZieglerRichterGebert}) determined by $S_1$ and $S_2$. The minimal Radon partitions in $S_1$ correspond to the circuits of the oriented matroids and therefore are preserved under $\sigma$. Thus $\conv(\sigma(P_{i_1}'))\cap \conv(\sigma(P_{i_2}')) \neq \emptyset$. The reverse implication is shown by the reasoning applied to $\sigma^{-1}$.
\end{proof}

%

As we mentioned in the Introduction, the graph $K$ in Figure~\ref{fig:test1} is $2$-partition induced (in particular by the partitioned point set in Figure~\ref{fig:test2}), but  not $2$-Tverberg, as implied by the following lemma:
\begin{lemma}\label{badforconvex}
Suppose $S$ is any set of points in convex position in $\R^2$. Then the graph $K$ in Figure~\ref{fig:test1} is not partition induced on $S$.
\end{lemma}
\begin{figure}
\centering
\begin{minipage}{.5\textwidth}
  \centering
  \includegraphics[width=.7\linewidth]{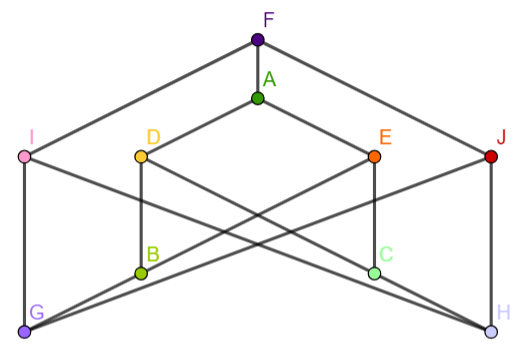}
  \caption{Graph $K$}
  \label{fig:test1}
\end{minipage}%
\begin{minipage}{.5\textwidth}
  \centering
  \includegraphics[width=1\linewidth]{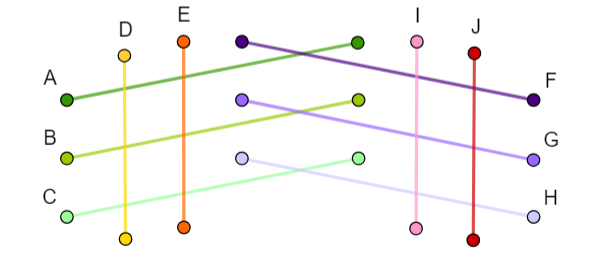}
  \caption{Partitioned point set with nerve $K$}
  \label{fig:test2}
\end{minipage}
\end{figure}

\begin{proof}
We note that since $K$ is a triangle free graph, it suffices to show that it is not the intersection graph of any partition of points in convex position. 
We argue by contradiction. Suppose that there is a set of points in convex position partitioned so that they have the graph above as their intersection graph. 
By Lemma~\ref{ordertype} we may assume the points are arranged on the boundary of a disc $\mathcal{D}$. Denote the convex hull of 
the points corresponding to each node $i$ by region $i$. In the rest of the proof of Lemma~\ref{badforconvex}, we will rely on the following.

\begin{claim} Consider the independent set of nodes $\{A,B,C\}$ in Figure~\ref{fig:test1}. Up to exchanging their labels (note that the graph is symmetric about $A,B,C$), there are 
two possible arrangements of the regions $A$,$B$, and $C$, pictured in Figures~\ref{convexcase1} and \ref{convexcase2}.

\begin{figure}
\centering
\begin{minipage}{.5\textwidth}
  \centering
  \includegraphics[width=.3\linewidth]{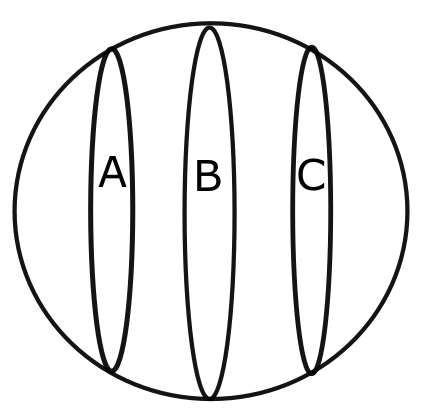}
  \label{convexcase1}
\end{minipage}%
\begin{minipage}{.5\textwidth}
  \centering
  \includegraphics[width=.3\linewidth]{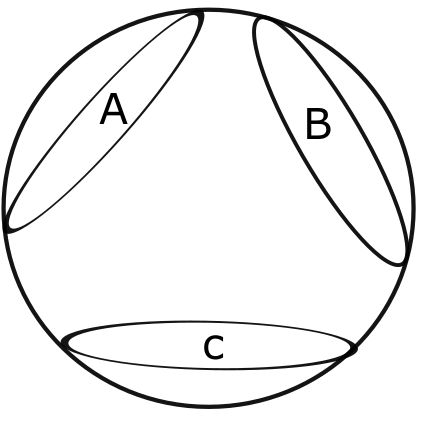}
  \label{convexcase2}
\end{minipage}
\end{figure}
\end{claim}
\begin{proof}[Proof of the claim]

The region $\mathcal{M}-B$ has two connected components.
If regions $A$ and $C$ lie in different connected components of $\mathcal{M}-B$, then regions $A,B,$ and $C$ must be arranged as in Figure~\ref{convexcase1}. 
Otherwise, $A$ and $C$ lie in the same connected component, say $\mathcal{N}$, of $\mathcal{M}-B$. If we walk clockwise around  the boundary of $\mathcal{N}$, 
we can only alternate twice between being in regions $A$ and $C$, reducing to the two possibilities shown.
\end{proof}

By the claim, we see that $A,B,$ and $C$ must be arranged (up to symmetry) as in one of the two cases pictured above. If they are arranged as in Figure~\ref{convexcase1}, note that regions $E$ and $F$ both intersect regions $A,B,$ and $C$. In that case it is easy to see that regions $E$ and $F$ must 
intersect, which is a contradiction.

If the regions are arranged as in Figure~\ref{convexcase2}, consider that regions $D, F, G,$ and $H$. Note that region $D$ intersects regions $A,B,C$. 
Also, region $F$ is disjoint from all the regions $B$ through $H$, while intersecting region $A$.
Similarly, region $G$ is disjoint from all the regions $A$ through $H$ except $B$. Also region $H$ is disjoint from all the regions $A$ through $H$ except $C$. Considering the two cases: $F,G,H$ lie in the same connected component of $\mathcal{M} - D$,  or $F,G,H$ lie in different connected components of $\mathcal{M}-D$, it is easy to see that, in both cases,
$F,G,$ and $H$ must be arranged as $A,B,$ and $C$ are in Figure~\ref{convexcase2}. Then $I,J$ are disjoint but both intersect $F,G$ and $H$, which is a contradiction by the argument above. Thus $K$ cannot be the nerve of a set of points in convex position.
\end{proof}

\end{document}